\documentclass[a4paper,11pt]{article}
\usepackage[utf8]{inputenc}
\usepackage[T1]{fontenc}
\usepackage{lmodern}
\usepackage{microtype}
\usepackage[a4paper,textwidth=16cm,textheight=22cm,centering]{geometry}

\usepackage{mathtools}
\usepackage{amssymb,amsfonts}
\usepackage{amsthm}
\usepackage{amsmath}
\usepackage{setspace}
\usepackage{indentfirst}

\usepackage{cite}
\usepackage{xcolor}
\usepackage[ ]{hyperref}
\hypersetup{
	colorlinks=true,   
	linktoc=all,        
	linkcolor=red,     
	citecolor=blue,} 

\theoremstyle{plain}
\newtheorem{theorem}{Theorem}[section]
\newtheorem{lemma}[theorem]{Lemma}

\newtheorem{corollary}[theorem]{Corollary}

\newtheorem{conjecture}[theorem]{Conjecture}

\theoremstyle{definition}

\theoremstyle{remark}

\begin{document}
	\date{}
	\begin{spacing}{1.03}
		\title{{The Erd\H{o}s--Hajnal hypergraph Ramsey problem for $r_4(6,n)$
        }}
\author{
Longma Du$^1$, \;\;\; Xinyu Hu$^2$, \;\;\; Ruilong Liu$^1$, \;\;\; Guanghui Wang$^1$}

\footnotetext[1]{School of Mathematics, Shandong University, Jinan 250100, P.~R.~China.
Emails: {\tt 202520303@mail.sdu.edu.cn} (L. Du),
{\tt liuruilong@mail.sdu.edu.cn} (R. Liu),
{\tt ghwang@sdu.edu.cn} (G. Wang).
Supported by the Natural Science Foundation of China (12231018) and State Key Laboratory of Cryptography and Digital Economy Security.}
\footnotetext[2]{Data Science Institute, Shandong University, Jinan 250100, P.~R.~China.
Email: {\tt huxinyu@sdu.edu.cn} (X. Hu).
Supported by National Postdoctoral Fellowship Program (C-tier) (GZC20252005) and NSFC (No. 12601677).}

\maketitle	

\begin{abstract}
The Ramsey number $r_k(s,n)$ is the smallest integer $N$ such that every $N$-vertex $k$-graph contains either a copy of $K_s^{(k)}$ or an independent set of size $n$. Erd\H{o}s and Hajnal conjectured that for every fixed $s>k\ge 4$, one has $r_k(s,n)\ge \operatorname{twr}_{k-1}(\Omega(n))$. This conjecture was independently verified by Mubayi and Suk, and by Conlon, Fox and Sudakov, for $k\ge4$ and $s\ge k+3$.
In this paper, we prove that $r_4(6,n)\ge 2^{2^{cn}}$ for some absolute constant $c>0$, improving upon our previous bound. Consequently, we confirm the Erd\H{o}s--Hajnal conjecture for $r_k(k+2,n)$ for all fixed $k\ge4$.
\end{abstract}

\section{Introduction}
A $k$-uniform hypergraph $H$ ($k$-graph for short) is a pair consisting of a set of vertices $V(H)$ and a collection of $k$-element subsets of $V(H)$. Let $K_n^{(k)}$ be the complete $k$-graph on $n$ vertices. The off-diagonal Ramsey number $r_k(s,n)$ \cite{R}, with $k$ and $s$ regarded as fixed parameters, is the smallest integer $N$ such that every $N$-vertex $k$-graph contains either a copy of $K_s^{(k)}$ or an independent
set of size $n$. 

The Ramsey numbers have been extensively studied since 1935 \cite{E-S-1}, with many classical results \cite{A-K-S-1,C-G-E,E-H-Con,E-R-2,K-1,L-R-Z,M-S-3,M-S-4,S-1,P-G-M,SP-1}.
Recent years have witnessed many breakthroughs, especially in graphs \cite{H-M-S,M-S-X,M-V-2,Bradac26,C-G-M-S,G-N-N-W}. Here we focus on the off-diagonal hypergraph setting.

For $3$-graphs, Conlon, Fox and Sudakov \cite{C-F-S-3} proved that for every $s\ge4$, $2^{\Omega(n\log n)}\le r_3(s,n)\le 2^{O(n^{s-2}\log n)}$, which improved the upper bound of Erd\H{o}s and Rado \cite{E-R-2} and the lower bound of Erd\H{o}s and Hajnal \cite{E-H-Con}.  
For $s>k\ge 4$, it is known that $r_k(s,n)\le\operatorname{twr}_{k-1} (n^{O(1)})$ \cite{E-R-2}, where the tower function $\operatorname{twr}_{k}(x)$ is defined by $\operatorname{twr}_{1}(x)=x$ and $\operatorname{twr}_{i+1}(x)=2^{\operatorname{twr}_{i}(x)}$. The Erd\H{o}s--Hajnal stepping-up lemma implies that $r_k(s,n)\ge \operatorname{twr}_{k-1} (\Omega(n))$ for $k\ge 4$ and for all $s\ge 2^{k-1}-k+3$. A fundamental conjecture about $r_k(s,n)$ was proposed by Erd\H{o}s and Hajnal \cite{E-H-Con}.
\begin{conjecture}[Erd\H{o}s and Hajnal \cite{E-H-Con}]\label{E-H-C}
For every fixed $s>k\ge4$, 
$r_k(s,n)\ge \operatorname{twr}_{k-1}(\Omega(n))$. 
\end{conjecture} 
Erd\H{o}s and Hajnal (see \cite{G-R-S-1}) 
showed that $r_4(7,n)\ge 2^{2^{\Omega(n)}}$. In \cite{C-F-S-1}, Conlon, Fox and Sudakov modified the Erd\H{o}s--Hajnal stepping-up lemma to show that Conjecture \ref{E-H-C} holds for all $s\ge \lceil\frac{5k}{2}\rceil-3$. Mubayi and Suk \cite{M-S-4}, as well as Conlon, Fox and Sudakov, independently verified Conjecture \ref{E-H-C} for $k\ge 4$ and $s\ge k+3$. Mubayi and Suk \cite{M-S-3} established the lower bounds $r_k(k+1,n)\ge \operatorname{twr}_{k-2} (n^{\Omega(\log n)})$ and $r_k(k+2,n)\ge \operatorname{twr}_{k-1} ({\Omega(n^{1/5})})$. Recently, the authors \cite{D-H-L-W-2} improved the lower bound for $r_k(k+2,n)$ to $\operatorname{twr}_{k-1}(\Omega(n^{1/2}))$ by proving $r_4(6,n)\ge 2^{2^{\Omega(n^{1/2})}}$. Subsequently, the authors \cite{D-H-L-W-3} showed that $r_k(k+1,n)\ge \operatorname{twr}_{k-1}(\Omega(n^{1/7}))$, thereby determining the tower growth rate. Shortly thereafter, Fan, Li, Lin, and Ning \cite{F-L-L-N}, as well as the authors, independently improved the lower bound to $r_k(k+1,n)\ge \operatorname{twr}_{k-1}(\Omega(n^{1/5}))$.
\smallskip

In this paper, we further improve the lower bound for $r_4(6,n)$.

\begin{theorem}\label{center-2}
     For all $n\ge 5$, $r_4(6,n)\ge 2^{2^{cn}}$, where $c>0$ is an absolute constant.
\end{theorem}

A variant of the classical Erd\H{o}s--Hajnal stepping-up lemma, due to Mubayi and Suk \cite[Lemma~2.1]{M-S-3}, states that for $k\ge 5$ and $n\ge s\ge k+1$, one has
\[
r_k(s,2kn) > 2^{\,r_{k-1}(s-1,n)-1}.
\]

Applying this lemma yields the following corollary, which establishes Conjecture~\ref{E-H-C} in the case $s=k+2$ for $k\ge 4$, leaving only the case $r_4(5,n)\ge 2^{2^{cn}}$ open.

\begin{corollary}\label{coro}
    For each $k\ge 4$ and all sufficiently large $n$,
$r_k(k+2,n)\ge \operatorname{twr}_{k-1}(cn)$,
where $c=c(k)>0$ is a constant.
\end{corollary}
\medskip

\section{Properties of the stepping-up technique}\label{sec2}

Fix a positive integer $D$ and let $V = \{0, 1, \ldots, 2^D - 1\}$. For each $v \in V$, write
$v = \sum_{i=0}^{D-1} v(i)2^i$, where $v(i) \in \{0, 1\}$ for every $i$.
For distinct $u,v\in V$, let
$\delta(u,v)$ denote the largest $i\in\{0,1,\ldots,D-1\}$ such that
$u(i)\neq v(i)$.

We use $\langle v_1, v_2,\ldots, v_r\rangle$ to denote an ordered vertex set with $v_1< v_2<\cdots<v_r$. For an ordered vertex set $S=\langle v_1, v_2,\ldots, v_r\rangle$, we also write
$\delta(S)=\delta(v_1,\ldots,v_r)
=(\delta(v_1,v_2),\ldots,\delta(v_{r-1},v_r))=(\delta_1,\ldots,\delta_{r-1})=(\delta_i)_{i=1}^{r-1}$, where $\delta_i=\delta(v_i,v_{i+1})$ for $i\in[r-1]$. 
For convenience, if inequalities are known between consecutive $\delta$-values, this will be indicated in the sequence by replacing the comma with the respective sign. For instance, assume that $S=\langle v_1,\ldots, v_5\rangle$ and $\delta_1< \delta_2 >\delta_3< \delta_4$. Then, since $\delta(v_1,v_2,v_3,v_4)=(\delta_1,\delta_2,\delta_3)$ satisfies $\delta_1< \delta_2 >\delta_3$, we write
$\delta(v_1,v_2,v_3,v_4)=(\delta_1<\delta_2>\delta_3)$. 
Similarly, when only some of the consecutive inequalities are known, as is the case for $\delta(v_1,v_2,v_4,v_5)$, we write $\delta(v_1,v_2,v_4,v_5)=(\delta_1<\delta_2~,~\delta_4)$.

For $2\le i\le r-2$, we say that $\delta_i$ is a \emph{local minimum} if
$\delta_{i-1}>\delta_i<\delta_{i+1}$, a \emph{local maximum} if $\delta_{i-1}<\delta_i>\delta_{i+1}$, and a \emph{local extremum} if it is either a local minimum or a local maximum. We call $\delta_i$ a \emph{local monotone} if $\delta_{i-1}<\delta_i<\delta_{i+1}$ or $\delta_{i-1}>\delta_i>\delta_{i+1}$. We say $\delta_1,\ldots,\delta_{r-1}$
form a monotone sequence if $\delta_1<\cdots<\delta_{r-1}$ (monotone increasing) or $\delta_1>\cdots>\delta_{r-1}$
(monotone decreasing), i.e., there is no local extremum. 

We have the following stepping-up properties (see \cite{G-R-S-1}).\medskip

\textbf{Property I.} For every triple $u < v < w$, $\delta(u,v) \neq \delta(v,w)$.
\medskip

Since $\delta_{i-1}\neq \delta_i$ for every $i$, every nonmonotone sequence $(\delta_i)_{i=1}^{r-1}$ has a local extremum.
\medskip

\textbf{Property II.} For $v_1 < \cdots < v_r$, $\delta(v_1,v_r) = \max_{1 \leq j \leq r-1} \delta(v_j,v_{j+1})$.
\medskip

We also have the following properties from Properties I and II (see \cite{F-H-L-L,H-L-L-W-1,M-S-3}).
\medskip

\textbf{Property III.} For $\delta(v_1,v_r) = \max_{1 \leq j \leq r-1} \delta(v_j,v_{j+1})$, there is a unique $\delta_i$ which achieves the maximum.
\medskip 

\textbf{Property IV.} For every 4-tuple $v_1 < \cdots < v_4$, if $\delta(v_1,v_2) > \delta(v_2,v_3)$, then $\delta(v_1,v_2) \neq \delta(v_3,v_4)$. Note that if $\delta(v_1,v_2) < \delta(v_2,v_3)$, it is possible that $\delta(v_1,v_2) = \delta(v_3,v_4)$.
\medskip

\textbf{Property V.} For $v_1 < \cdots < v_r$, suppose that $\delta_1,\ldots,\delta_{r-1}$ form a monotone sequence. Then for every subset of $k$ vertices $v_{i_1} < \cdots < v_{i_k}$, $\delta(v_{i_1},v_{i_2}),\delta(v_{i_2},v_{i_3}),\ldots,\delta(v_{i_{k-1}},v_{i_k})$ form a monotone sequence. Moreover, for every subset of $k-1$ such $\delta_j$'s, i.e. $\delta_{j_1},\delta_{j_2},\ldots,\delta_{j_{k-1}}$, there are $k$ vertices $v_{i_1},\ldots,v_{i_k}$ such that $\delta(v_{i_t},v_{i_{t+1}}) = \delta_{j_t}$.
\medskip

\section{The lower bound for $r_4(6,n)$}

In this section, we begin by considering a graph coloring with certain properties, which will later be used to define the edge set of a $4$-graph. We omit its proof, since it follows from a probabilistic argument analogous to that used in the proof of Lemma~3.1 in \cite{D-H-L-W-2,D-H-L-W-3}.

\begin{lemma}\label{phi}
There exists an absolute constant $c_0>0$ such that for every $n\ge5$ the following holds. There is a red/blue coloring $\phi$ of the pairs of $\{0,1,\ldots,\lfloor 2^{c_0n}\rfloor-1\}$ with the property that every $n$-set $A\subset \{0,1,\ldots,\lfloor 2^{c_0n}\rfloor-1\}$ contains a $3$-tuple $a_i<a_j<a_k$ satisfying $$\phi(a_i,a_j)=\phi(a_j,a_k)=\text{blue}~\text{and}~ \phi(a_i,a_k)=\text{red}.$$
\end{lemma}
\medskip

Let $c_0>0$ be the constant from Lemma \ref{phi}, and let $U=\{0,1,\ldots,\lfloor2^{c_0n}\rfloor-1\}$ and $\phi:\binom{U}{2}\to \{\text{red},\text{blue}\}$ be a $2$-coloring of the pairs of $U$ satisfying the property given in the lemma. Now, let $N=2^{\lfloor2^{c_0n}\rfloor}$ and $V(H)=\{0,1,\ldots,N-1\}$. Then we shall use the coloring
$\phi$ to produce a $K^{(4)}_6$-free $4$-graph $H$ on $V(H)$ with $\alpha(H)<4n-5$ as follows. For any $4$-tuple $e=\langle v_1,v_2,v_3,v_4 \rangle$ of $V(H)$, set $e\in E(H)$ if and only if one of the following holds:
\begin{enumerate}
\item[\textbf{(i)}] $\delta(e)$ is monotone, $\phi(\delta_1,\delta_2)=\phi(\delta_2,\delta_3)=\text{blue}$, and $\phi(\delta_1,\delta_3)=\text{red}$; 
\item[\textbf{(ii)}] $\delta_1>\delta_2<\delta_3$ and $\phi(\delta_1,\delta_2)=\phi(\delta_2,\delta_3)=\text{blue}$;
\item[\textbf{(iii)}] $\delta_1<\delta_2>\delta_3$ and at least one of $\phi(\delta_1,\delta_2)$ and $\phi(\delta_2,\delta_3)$ is red.
\end{enumerate}

\subsection{$H$ is $K_6^{(4)}$-free}

In this subsection, we show that $H$ is $K_6^{(4)}$-free. To see this, suppose to the contrary that $P = \langle v_1, \ldots, v_6 \rangle$ induces a $K_6^{(4)}$ in $H$. This will lead to a contradiction. Recall that $\delta(P)=(\delta_1,\delta_2,\delta_3,\delta_4,\delta_5)$. Let $\delta_k$ denote the unique largest element in $\delta(P)$. The uniqueness of $\delta_{k}$ follows from Property III.

If $k\in\{1,2\}$, we consider $\delta(v_k,v_3,v_4,v_5,v_6)=(\delta_k>\delta_3,\delta_4,\delta_5)$. Since $P$ induces a $K_6^{(4)}$ in $H$, the 4-tuples $(v_k,v_3,v_4,v_5)$, $(v_k,v_4,v_5,v_6)$ and $(v_3,v_4,v_5,v_6)$ are edges of $H$. From the edges $(v_k,v_3,v_4,v_5)$ and $(v_k,v_4,v_5,v_6)$, together with $\delta_k>\max\{\delta_3,\delta_4\}$ and $\delta_k>\max\{\delta_4,\delta_5\}$ and \textbf{(i)}--\textbf{(ii)}, it follows that $\phi(\delta_k,\delta_3)=\phi(\delta_k,\delta_4)=\text{blue}$. If $\delta_3>\delta_4$, then \textbf{(i)} implies $\phi(\delta_k,\delta_4)=\text{red}$, a contradiction. Thus $\delta_3<\delta_4$ by Property I. Moreover, applying \textbf{(ii)} to $(v_k,v_3,v_4,v_5)$ yields $\phi(\delta_3,\delta_4)=\text{blue}$. If $\delta_4>\delta_5$, then $\delta(v_3,v_4,v_5,v_6)=(\delta_3<\delta_4>\delta_5)$. It follows from $\phi(\delta_3,\delta_4)=\text{blue}$ and \textbf{(iii)} that $\phi(\delta_4,\delta_5)=\text{red}$, which implies that $(v_k,v_4,v_5,v_6)$ is not an edge by \textbf{(i)}, a contradiction. Therefore $\delta_3<\delta_4<\delta_5$. Applying \textbf{(i)} to $(v_3,v_4,v_5,v_6)$ gives $\phi(\delta_3,\delta_5)=\text{red}$. This, together with $\delta(v_k,v_3,v_4,v_6)=(\delta_k>\delta_3<\delta_5)$, implies by \textbf{(ii)} that $(v_k,v_3,v_4,v_6)$ is not an edge, again a contradiction. The same argument also gives a contradiction when $k\in\{4,5\}$.

It remains to check $k=3$. Since $(v_1,v_2,v_3,v_4)$ and $(v_3,v_4,v_5,v_6)$ are edges of $H$, and $\delta_3>\delta_2$, it follows from \textbf{(i)} and \textbf{(ii)} that $\phi(\delta_2,\delta_3)=\text{blue}$. Likewise, \textbf{(i)}--\textbf{(ii)} implies that $\phi(\delta_3,\delta_4)=\text{blue}$. However, by \textbf{(iii)}, $(v_2,v_3,v_4,v_5)$ is not an edge, a contradiction.

This completes the proof that $H$ is $K_6^{(4)}$-free.

\subsection{$\alpha(H)<4n-5$}
Now we show that $\alpha(H)<4n-5$. Suppose that there is a set $Q=\langle v_1,v_2,\ldots ,v_m\rangle$ of $m$ vertices that induces an independent set in $H$. Recall that $\delta_i=\delta(v_i,v_{i+1})$, and hence $\delta(Q)=(\delta_{i})_{i=1}^{m-1}$. 

\begin{lemma}\label{no-mono-n}
There is no monotone subsequence $(\delta_{i_\ell})_{\ell=1}^n\subset (\delta_j)_{j=1}^{m-1}$ such that for any $a,b,c \in [n]$ with $a<b<c$, there exists $\{u_1,\ldots,u_4\}\subset\{v_1,\ldots,v_m\}$ such that $\delta(u_1,\ldots,u_4)=(\delta_{i_a},\delta_{i_b},\delta_{i_c})$.
\end{lemma}

\noindent\emph{Proof of Lemma \ref{no-mono-n}.~}
We present the argument for the increasing case. The decreasing case is similar. Suppose to the contrary that $(\delta_{i_\ell})_{\ell=1}^n$ is such a monotone increasing subsequence. It follows from Lemma \ref{phi} that there is a $3$-tuple $\delta_{i_a},\delta_{i_b},\delta_{i_c}$ with $\delta_{i_a}<\delta_{i_b}<\delta_{i_c}$ such that $$\phi(\delta_{i_a},\delta_{i_b})=\phi(\delta_{i_b},\delta_{i_c})=\text{blue}~\text{and}~\phi(\delta_{i_a},\delta_{i_c})=\text{red}.$$
Since there exists $\{u_1,\ldots,u_4\}\subset \{v_1,\ldots,v_m\}$ such that $\delta(u_1,\ldots,u_4)=(\delta_{i_a},\delta_{i_b},\delta_{i_c})$ from the assumption, $\{u_1,\ldots,u_4\}$ forms an edge by \textbf{(i)}, a contradiction.\hfill$\Box$
\medskip

We claim that there is at most one local maximum in the sequence $\delta(Q)$. Suppose, to the contrary, that there are at least two local maxima. Then we can choose two consecutive local maxima, say $\delta_i$ and $\delta_j$ with $2\le i<j\le m-2$, and there is exactly one local minimum between them, say at some index $\ell$ with $i<\ell<j$. By Property II, we obtain that
\[
\delta(v_{i-1},v_i,v_\ell,v_{\ell+1},v_{j+1},v_{j+2})=(\delta_{i-1}<\delta_i>\delta_\ell<\delta_j>\delta_{j+1}).
\]
Since $(v_{i-1},v_i,v_\ell,v_{\ell+1})$ and $(v_\ell,v_{\ell+1},v_{j+1},v_{j+2})$ are nonedges in $H$, it follows from \textbf{(iii)} that $\phi(\delta_i,\delta_\ell)=\phi(\delta_\ell,\delta_j)=\text{blue}$. But then \textbf{(ii)} implies that $(v_i,v_\ell,v_{\ell+1},v_{j+1})$ is an edge in $H$, a contradiction. Therefore, there is at most one local maximum in $\delta(Q)$.

Since $\delta(Q)$ has at most one local maximum and, by Lemma \ref{no-mono-n}, contains no monotone consecutive subsequence of length $n$, we have $m-1\le 4(n-1)-3$. Hence $m\le 4n-6$.
This completes the proof.
\medskip

Finally, we present the proof of Theorem \ref{center-2}.

\noindent\textit{Proof of Theorem \ref{center-2}.~} Combining the two preceding subsections, we have constructed a $4$-graph $H$ on 
$N=2^{\lfloor 2^{c_0n}\rfloor}$ vertices such that $H$ is $K_6^{(4)}$-free and 
$\alpha(H)<4n-5$. Therefore 
$r_4(6,4n-5)>2^{\lfloor 2^{c_0n}\rfloor}$. Recall that $c_0>0$ is a sufficiently small absolute constant. Consequently, there exists an absolute constant $c>0$ 
such that $r_4(6,n)\ge 2^{2^{cn}}$ for all $n\ge 5$. This completes the proof of 
Theorem~\ref{center-2}. \hfill$\Box$
\medskip


\section*{Declaration on the use of AI}
The authors acknowledge the use of AI tools. We used ChatGPT to discuss constructions. Specifically, ChatGPT enabled us to adjust how the edges are defined in the constructed \(4\)-graph, thereby simplifying the analysis of local extrema in the associated \(\delta\)-sequence, whereas our previous approach relied on a multi-layer analysis of local maxima. Along these lines, we find that several related constructions yield the same lower bound. In this paper, we adopt a simple construction and omit the others, including the original construction generated by ChatGPT. All mathematical arguments and proofs in the manuscript were written and checked by the authors.

\end{spacing}
\end{document}